\documentclass[11pt]{amsart}

\usepackage[margin=.9in]{geometry}
\usepackage[T1]{fontenc}
\usepackage[utf8]{inputenc}
\usepackage{lmodern}
\usepackage{amsmath,amssymb,amsthm,mathtools}
\usepackage[shortlabels]{enumitem}
\usepackage{tikz}
\usetikzlibrary{decorations.pathreplacing,decorations.pathmorphing,arrows.meta}
\usepackage{caption}
\PassOptionsToPackage{hyphens}{url}
\usepackage[colorlinks=true,linkcolor=blue,citecolor=blue,urlcolor=blue]{hyperref}
\hypersetup{
  pdftitle={A Direct Stein Proof of the Burke Property for the Stationary O'Connell--Yor Polymer},
  pdfauthor={Ruixuan Zhang},
  pdfsubject={Stationary O'Connell--Yor polymer and the Burke property},
  pdfkeywords={O'Connell--Yor polymer, Burke property, Gaussian integration by parts, Stein's method}
}

\allowdisplaybreaks
\AtBeginDocument{%
  \setlength{\abovedisplayskip}{3pt plus 2pt minus 1pt}%
  \setlength{\belowdisplayskip}{3pt plus 2pt minus 1pt}%
  \setlength{\abovedisplayshortskip}{1.5pt plus 1pt}%
  \setlength{\belowdisplayshortskip}{2.5pt plus 1pt minus 1pt}%
  \setlength{\jot}{1.25pt}%
}
\numberwithin{equation}{section}
\numberwithin{figure}{section}

\newtheorem{theorem}{Theorem}[section]
\newtheorem{corollary}[theorem]{Corollary}
\newtheorem{lemma}[theorem]{Lemma}
\newtheorem{proposition}[theorem]{Proposition}
\theoremstyle{definition}

\theoremstyle{remark}
\newtheorem{remark}[theorem]{Remark}

\newcommand{\R}{\mathbb{R}}
\newcommand{\Z}{\mathbb{Z}}
\newcommand{\EE}{\mathbb{E}}

\newcommand{\DD}{\mathbb{D}}
\newcommand{\ind}{\mathbf{1}}
\newcommand{\cB}{\mathcal{B}}
\newcommand{\cK}{\mathcal{K}}
\newcommand{\cV}{\mathcal{V}}
\newcommand{\cZ}{\mathcal{Z}}
\newcommand{\be}{\beta}
\newcommand{\eps}{\varepsilon}
\newcommand{\fH}{\mathfrak{H}}
\newcommand{\ga}{\gamma}
\newcommand{\DN}{\Delta_n}
\newcommand{\Var}{\text{Var}}
\newcommand{\OCY}{\textup{O'Connell--Yor}}
\DeclareMathOperator{\sgn}{sgn}

\tikzset{
  brownian background/.style={
    color={rgb,255:red,155;green,155;blue,155},
    draw opacity=.66,
    line cap=round,
    decorate,
    decoration={random steps,segment length=7pt,amplitude=1.4pt}
  },
  >=Latex,
  bluepath/.style={line width=1.1pt,blue}
}

\title[A Stein proof of the Burke property]
{A Stein Proof of the Burke Property\\
for the Stationary O'Connell--Yor Polymer}

\author[R.~Zhang]{Ruixuan Zhang}
\address{R. Zhang, Department of Mathematics, University of Utah,
Salt Lake City, UT 84112, USA}
\email{ray.zhang@math.utah.edu}

\date{September 15, 2026}
\keywords{
O'Connell--Yor polymer
$\cdot$ Burke property
$\cdot$ Gaussian integration by parts
$\cdot$ Stein's method
}

\begin{document}

\begin{abstract}
We give a direct proof of the Burke property for the stationary O’Connell–Yor polymer using Stein's method. Building on a free-energy identity developed in joint work with Firas Rassoul-Agha, Xiao Shen, and Guangqu Zheng, we establish a multilevel extension that yields a joint Gaussian–Gamma Stein identity, from which the full Burke property follows. The proof is based on a Gaussian integration-by-parts framework developed in that work. Passing to zero temperature gives the corresponding result for stationary point-to-line Brownian last-passage percolation.
\end{abstract}

\maketitle
\enlargethispage{7pt}

\section{Introduction}

Among \((1+1)\)-dimensional random growth models in the
Kardar--Parisi--Zhang (KPZ) universality class, the \(\OCY\) polymer of O'Connell
and Yor \cite{OCY01} is one of the fundamental exactly solvable directed
polymer models.  Its stationary version satisfies a Brownian analogue of
Burke's theorem.  Along any down-right path, the horizontal free-energy
increments have the law of independent Brownian increments, the vertical
partition-function ratios have Gamma laws, and the two families are
independent.

The Burke property is an important structural input in the study of
stationary measures and Busemann processes in KPZ-related models.  Its explicit
marginal laws and independence statements make the stationary boundary data
tractable.  This structure has been used to prove fluctuation bounds
\cite{SV10}, to construct and classify semi-infinite polymer measures
\cite{ARS20,JanjigianRassoulAgha20}, and to analyze coalescence and the global
geometry of semi-infinite geodesics in Brownian last-passage percolation and
the directed landscape \cite{SeppalainenSorensen23,
BusaniSeppalainenSorensen24}.  It also underlies the description of joint
stationary measures and Busemann processes for the KPZ equation
\cite{GroathouseRassoulAghaSeppalainenSorensen25}, as well as independence
estimates for Busemann increments and polymer endpoints \cite{Shen25}.

The original proof is based on the generalized Brownian queue of O'Connell
and Yor \cite{OCY01}.  It combines the invariance of the Matsumoto--Yor path
transformation \cite{MY01} with Dufresne's identity \cite{Duf01}.  We take a
different route.  Our argument is probabilistic, based on Gaussian integration
by parts, and does not rely on the algebraic path-transform structure underlying the classical proof.  Starting from
the free-energy identity of \cite{RassoulAghaShenZhangZheng26}, we extend the
identity to polymers started
from different levels and obtain a joint Gaussian--Gamma Stein identity at a
single corner of a down-right path.  This identifies the horizontal and
vertical increments at that corner, including their independence.  Peeling
off the corners one at a time then gives the full Burke property.  The same
description survives the zero-temperature limit and yields the corresponding
result for stationary point-to-line Brownian last-passage percolation. As a further consequence, the diagonal free-energy identity connects the free-energy variance with polymer exit geometry, recovering the Seppäläinen–Valkó variance identity \cite[Theorem~3.6]{SV10} in differential form and yielding its zero-temperature analogue for stationary Brownian last-passage percolation.

Section~\ref{sec:model} introduces the model and its general-level version.
Section~\ref{sec:main} states the multilevel identity, derives the one-corner
Stein identity, and proves the Burke properties.  Section~\ref{sec:proof}
contains the proofs of the multilevel and one-corner identities. 

\section{Model and notation}
\label{sec:model}

We begin with the stationary point-to-line polymer started from level zero.
The general-level version introduced afterward is needed to compare free
energies attached to different horizontal pieces of a down-right path.

\subsection{Stationary point-to-line polymers}

Fix a terminal level \(n\in\Z_{>0}\) and an inverse temperature
\(\be>0\).  We use the jump times of a rate-\(\be\) Poisson process to
parametrize polymer paths.  Set \(\ga_0=0\), write
\(\DN=\{0<\ga_1<\cdots<\ga_n\}\), and denote the law of the first \(n\)
jump times by \(\mu_\be\).  Thus
\begin{align}\label{eq:reference}
  \mu_\be(d\ga)
  =
  \be^ne^{-\be\ga_n}
  \ind_{\DN}(\ga)\,
  d\ga_1\cdots d\ga_n.
\end{align}

Let \(\{B_i\}_{i\geq0}\) be independent two-sided standard Brownian
motions with \(B_i(0)=0\).  The vector \(\ga\in\DN\) represents an
up-right path from \((x,0)\) to level \(n\).  On level \(i\), it travels
from \(x+\ga_i\) to \(x+\ga_{i+1}\), and its terminal point is
\((x+\ga_n,n)\).  Thus the spatial jump locations are
\(x+\ga_i\), \(1\leq i\leq n\).

The Hamiltonian collected by this path is
\begin{align}\label{eq:stationary-H}
  H_x(\ga)
  =
  B_n(x+\ga_n)
  +
  \sum_{i=0}^{n-1}
  \bigl(
    B_i(x+\ga_{i+1})-B_i(x+\ga_i)
  \bigr).
\end{align}
The corresponding point-to-line partition function, free
energy, and quenched polymer measure are
\begin{align}\label{eq:stationary-Z}
  Z_x^\be=
  \be^{-n}
  \int_{\DN}
    e^{\be H_x(\ga)}\,
    \mu_\be(d\ga),
  \qquad F_x^\be=\frac1\be\log Z_x^\be,
  \qquad
  Q_x^\be(d\ga)
  =
  \frac{\be^{-n}e^{\be H_x(\ga)}}
       {Z_x^\be}
  \mu_\be(d\ga).
\end{align}
The partition function is finite almost surely by
\cite[Lemma~A.2]{RassoulAghaShenZhangZheng26}.  Here and below, \(\EE\)
denotes expectation with respect to the Brownian environment.

\subsection{Polymers started from a general level}
To formulate the free-energy identity for polymers with different starting
points, we allow the \(\OCY\) polymer to start from an arbitrary level.  This
uses the same environment and changes only the set of levels traversed by the
path.  For \(0\leq k<n\), let
\(\Delta_{k,n}=\{0=\ga_k<\ga_{k+1}<\cdots<\ga_n\}\) and put
\begin{align}\label{eq:shifted-reference}
  \mu_\be^{k,n}(d\ga)=
  \be^{n-k}e^{-\be\ga_n}
  \ind_{\Delta_{k,n}}(\ga)\,
  d\ga_{k+1}\cdots d\ga_n.
\end{align}
For \(k=n\), set \(\Delta_{n,n}=\{\ga_n=0\}\) and
\(\mu_\be^{n,n}=\delta_0\).
A path under \(\mu_\be^{k,n}\) starts from \((x,k)\), enters level \(j\)
at \(x+\ga_j\), and terminates on level \(n\).  Its Hamiltonian, partition
function, free energy, and quenched measure are
\begin{align}
  &H_{(x,k),n}(\ga)=
  B_n(x+\ga_n)
  +\sum_{j=k}^{n-1}
    \bigl(
      B_j(x+\ga_{j+1})-B_j(x+\ga_j)
    \bigr),
  \label{eq:H}\\
  &Z_{(x,k),n}^{\be}=
  \be^{-(n-k)}
  \int_{\Delta_{k,n}}
    e^{\be H_{(x,k),n}(\ga)}
    \mu_\be^{k,n}(d\ga),
  \qquad F_{(x,k),n}^{\be}=\frac1\be\log Z_{(x,k),n}^{\be},
  \label{eq:Z}\\
  &Q_{(x,k),n}^{\be}(d\ga)=
  \frac{\be^{-(n-k)}e^{\be H_{(x,k),n}(\ga)}}
       {Z_{(x,k),n}^{\be}}
  \mu_\be^{k,n}(d\ga).
  \label{eq:Q}
\end{align}
At level zero these definitions reduce to those above:
\begin{align}\label{eq:level-zero-identification}
  Z_{(x,0),n}^{\be}
  =Z_x^\be,
  \qquad
  F_{(x,0),n}^{\be}
  =F_x^\be,
  \qquad
  Q_{(x,0),n}^{\be}
  =Q_x^\be.
\end{align}
At the other endpoint, when \(k=n\), there is no path integral and
\begin{align}\label{eq:top}
  Z_{(x,n),n}^{\be}=e^{\be B_n(x)},
  \qquad
  F_{(x,n),n}^{\be}=B_n(x),
  \qquad
  Q_{(x,n),n}^{\be}=\delta_0.
\end{align}

For \(1\leq i\leq n\), the vertical edge at \((x,i)\) carries the
free-energy increment and partition-function ratio
\begin{align}\label{eq:vertical}
  &\cV_x^\be(i)=F_{(x,i),n}^{\be}-F_{(x,i-1),n}^{\be},
  \qquad G_{x,i}^{\be}=
  e^{\be\cV_x^\be(i)}
  =
  \frac{Z_{(x,i),n}^{\be}}
       {Z_{(x,i-1),n}^{\be}}.
\end{align}
Throughout, \(\operatorname{Gamma}(a,\lambda)\) denotes the Gamma distribution with shape \(a\) and rate \(\lambda\), with density \(\lambda^ag^{a-1}e^{-\lambda g}\ind_{\{g>0\}}/\Gamma(a)\).

\section{Stein identities and the Burke property}
\label{sec:main}

We first record the variables carried by a down-right path and state the multilevel free-energy identity. A one-corner
Stein identity then separates the variables at the lowest corner from those
above it. Iterating this separation yields the Burke property.

\subsection{Variables along a down-right path}
\label{subsec:path-variables}
Fix \(x_n\leq x_{n-1}\leq\cdots\leq x_1\).  The associated down-right
path runs right on level \(n\) to \(x_n\), descends one level at each
\(x_k\), and reaches level zero at \(x_1\).  We record its horizontal
free-energy increments and vertical partition-function ratios by
\begin{align}
  \cB_n^\be(s)&=
F_{(x_n,n),n}^{\be}-F_{(x_n-s,n),n}^{\be},
\qquad s\geq0,\label{eq:top-process}\\
\cB_k^\be(s)&=
F_{(x_{k+1}+s,k),n}^{\be}
-F_{(x_{k+1},k),n}^{\be},
\qquad 0\leq s\leq x_k-x_{k+1},\quad 1\leq k<n,
\label{eq:middle-process}\\
\cB_0^\be(s)&=
F_{x_1+s}^\be-F_{x_1}^\be,
\qquad s\geq0,\label{eq:bottom-process}\\
G_i&=G_{x_i,i}^{\be},
\qquad 1\leq i\leq n.\label{eq:path-G}
\end{align}
See Figure~\ref{fig:Burke-path}.

\begin{center}
  \begin{minipage}{.88\linewidth}
  \centering
  \begin{tikzpicture}[x=.92cm,y=.72cm]
    \foreach \yy in {0,1,3,4}{
      \draw[gray!55] (0.3,\yy)--(9.4,\yy);
    }
    \node[left] at (0.3,0) {$0$};
    \node[left] at (0.3,1) {$1$};
    \node[left] at (0.3,2) {$\vdots$};
    \node[left] at (0.3,3) {$n-1$};
    \node[left] at (0.3,4) {$n$};

    \draw[bluepath]
      (0.55,4)--(1.85,4)--(1.85,3)--(3.25,3);
    \draw[bluepath,densely dotted]
      (3.25,3)--(4.10,3)--(4.10,2.55)
      --(4.85,2.55)--(4.85,2.00)
      --(5.55,2.00)--(5.55,1);
    \draw[bluepath]
      (5.55,1)--(7.25,1)--(7.25,0)--(9.15,0);

    \fill[blue] (1.85,4) circle (1.6pt);
    \fill[blue] (1.85,3) circle (1.6pt);
    \fill[blue] (3.25,3) circle (1.6pt);
    \fill[blue] (5.55,1) circle (1.6pt);
    \fill[blue] (7.25,1) circle (1.6pt);
    \fill[blue] (7.25,0) circle (1.6pt);

    \node[above right,font=\scriptsize] at (1.85,4) {$x_n$};
    \node[above right,font=\scriptsize] at (3.25,3) {$x_{n-1}$};
    \node[below left,font=\scriptsize] at (5.55,1) {$x_2$};
    \node[below right,font=\scriptsize] at (7.25,0) {$x_1$};

    \node[above,blue,font=\scriptsize] at (1.05,4)
      {$F_{(\,\cdot\,,n),n}^{\be}$};
    \node[above,blue,font=\scriptsize] at (2.55,3)
      {$F_{(\,\cdot\,,n-1),n}^{\be}$};
    \node[above,blue,font=\scriptsize] at (6.35,1)
      {$F_{(\,\cdot\,,1),n}^{\be}$};
    \node[below,blue,font=\scriptsize] at (8.25,0)
      {$F_{\,\cdot\,}^\be$};

    \node[left,font=\scriptsize] at (1.78,3.50)
      {$G_{x_n,n}^{\be}$};
    \node[right,font=\scriptsize] at (7.35,.50)
      {$G_{x_1,1}^{\be}$};
    \node at (4.65,2.20) {$\cdots$};
  \end{tikzpicture}
  \captionsetup{width=.84\linewidth,hypcap=false}
  
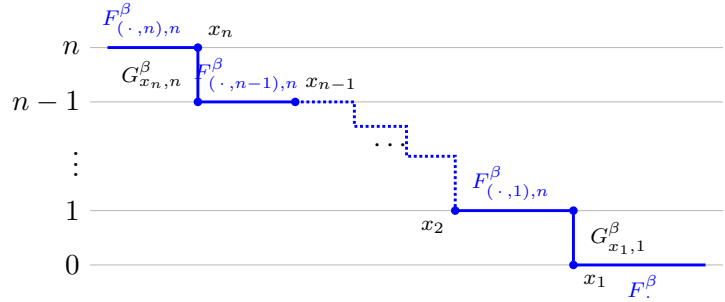
\captionof{figure}{A down-right path determined by
  \(x_n\leq\cdots\leq x_1\).  Horizontal edges carry increments of
  \(F_{(\,\cdot\,,k),n}^{\be}\), and the vertical edge at \(x_i\)
  carries \(G_{x_i,i}^{\be}\).}
  \label{fig:Burke-path}
  \end{minipage}
\end{center}

\subsection{The level-zero identity and its multilevel extension}

We begin with the stationary Gaussian Stein identity from
\cite[Proposition~1.8]{RassoulAghaShenZhangZheng26}. Its linear case is
the two-point correlation formula in
\cite[Corollary~1.5]{RassoulAghaShenZhangZheng26}.

\begin{proposition}[Stationary Gaussian--Stein identity]
\label{prop:level-zero-input}
Fix \(m\geq1\), let \(\mathbf y=(y_1,\ldots,y_m)\in\R^m\), and set
\(\mathbf F_{\mathbf y}^{\be}=(F_{y_1}^{\be},\ldots,F_{y_m}^{\be})\).
Suppose that \(\varphi\in C^1(\R^m)\) and
\(\|\nabla\varphi\|_\infty<\infty\).
Then the map
\(x\mapsto\EE[\varphi(\mathbf F_{\mathbf y}^{\be})F_x^{\be}]\) is
continuously differentiable on each connected component of
\(\R\setminus\{y_1,\ldots,y_m\}\).  Moreover, for
\(x\notin\{y_1,\ldots,y_m\}\),
\begin{align}\label{eq:level-zero-input}
  \frac{\partial}{\partial x}
  \EE\left[\varphi(\mathbf F_{\mathbf y}^{\be})F_x^{\be}\right]
  =
  \sum_{j=1}^{m}\EE\left[
    \partial_j\varphi(\mathbf F_{\mathbf y}^{\be})
    \left\{Q_x^\be\{\ga_n>-x\}-\ind_{\{x>y_j\}}\right\}
  \right].
\end{align}
\end{proposition}

Here \(Q_x^{\be}\{\ga_n>-x\}\) is the quenched probability that the
terminal point \((x+\ga_n,n)\) of the polymer started from \((x,0)\) lies
to the right of \((0,n)\).  The level-zero identity was used in
\cite[Corollary~1.10]{RassoulAghaShenZhangZheng26} to identify the
horizontal Busemann process as a two-sided standard Brownian motion.  The
multilevel identity below contains the level-zero formula as a special case
and extends it to free energies based at positive levels, including the
spatial diagonal.  In the proof of the Burke property below, the Brownian
finite-dimensional distributions are recovered from the resulting joint
Stein identity. \cite[Lemma~C.2]{RassoulAghaShenZhangZheng26} is needed only
to pass to continuous processes.

\begin{theorem}[Multilevel free-energy identity]
\label{thm:multilevel-identity}
Fix \(m\geq1\), let \(\mathbf y=(y_1,\ldots,y_m)\in\R^m\) and
\(\mathbf k=(k_1,\ldots,k_m)\in\{0,\ldots,n\}^m\), and set
\begin{align*}
  \mathbf F_{\mathbf y,\mathbf k}^{\be}
  =\bigl(F_{(y_1,k_1),n}^{\be},\ldots,F_{(y_m,k_m),n}^{\be}\bigr).
\end{align*}
Suppose that \(\varphi\in C^1(\R^m)\) and
\(\|\nabla\varphi\|_\infty<\infty\).
Set \(\mathcal D_0=\{y_j:k_j=0\}\).  The map
\(x\mapsto\EE[\varphi(\mathbf F_{\mathbf y,\mathbf k}^{\be})F_x^\be]\)
is continuously differentiable on each connected component of
\(\R\setminus\mathcal D_0\).

\smallskip
\noindent\textup{(i) Off the diagonal.}
If \(x\ne y_j\) for every \(1\leq j\leq m\), then
\begin{align}\label{eq:multilevel-identity}
  \frac{\partial}{\partial x}
  \EE\left[\varphi(\mathbf F_{\mathbf y,\mathbf k}^{\be})F_x^\be\right]
  =
  \sum_{j=1}^{m}
  \EE\left[
    \partial_j\varphi(\mathbf F_{\mathbf y,\mathbf k}^{\be})
    \left\{
      Q_x^\be\{\ga_n>-x\}
      -Q_x^\be\{\ga_{k_j}>y_j-x\}
    \right\}
  \right].
\end{align}
If \(x\in\{y_1,\ldots,y_m\}\setminus\mathcal D_0\), the right-hand side
has a continuous extension to \(x\), and this extension equals the derivative
there.

\begin{samepage}
\smallskip
\noindent\textup{(ii) On the diagonal.}
Suppose that \(x=y_r\) for some \(1\leq r\leq m\).
For \(\eps>0\), let a superscript \(\eps\) denote the even mollification
of the Brownian environment introduced in \eqref{eq:mollification}, and set
\(\mathbf F_{\mathbf y,\mathbf k}^{\be,\eps}=
\bigl(F_{(y_1,k_1),n}^{\be,\eps},\ldots,
F_{(y_m,k_m),n}^{\be,\eps}\bigr)\).
Then
\begin{equation}\label{eq:multilevel-diagonal}
\begin{aligned}
  &\lim_{\eps\downarrow0}\EE\left[
    \varphi(\mathbf F_{\mathbf y,\mathbf k}^{\be,\eps})
    \left.\partial_sF_s^{\be,\eps}\right|_{s=x}\right]\\[-2pt]
  &\qquad=\sum_{j=1}^{m}
  \EE\left[
    \partial_j\varphi(\mathbf F_{\mathbf y,\mathbf k}^{\be})
    \left\{
      Q_x^\be\{\ga_n>-x\}
      -Q_x^\be\{\ga_{k_j}>y_j-x\}
      -\frac12\ind_{\{y_j=x,\,k_j=0\}}
    \right\}
  \right].
\end{aligned}
\end{equation}
\end{samepage}
\end{theorem}

\begin{remark}
	For \(k_j=0\), \(Q_x^\be\{\ga_0>y_j-x\}=\ind_{\{x>y_j\}},\)
	so part~(i) reduces to Proposition~\ref{prop:level-zero-input} when
	\(k_1=\cdots=k_m=0\). 
	If \(x\in\{y_1,\ldots,y_m\}\setminus\mathcal D_0\), then every \(j\)
	with \(y_j=x\) has \(k_j>0\). Since the reference polymer satisfies
	\(\ga_k>0\) for every \(k>0\), \(Q_x^\be\{\ga_{k_j}>y_j-x\}=
	Q_x^\be\{\ga_{k_j}>0\}=1.\)
	Thus the right-hand side in part~(i) extends continuously through such
	positive-level coincidences and gives the ordinary derivative there.
	
	At a general coincidence \(y_j=x\), the total subtracted entrance term
	in \eqref{eq:multilevel-diagonal} is
	\[
	Q_x^\be\{\ga_{k_j}>0\}
	+\frac12\ind_{\{k_j=0\}}
	=
	\begin{cases}
		1, & k_j>0,\\[2pt]
		\frac12, & k_j=0.
	\end{cases}
	\]
	Hence, if \(k_j>0\) for every \(j\) with \(y_j=x\), the limit in
	part~(ii) agrees with the ordinary derivative from part~(i). At a
	level-zero coincidence, even mollification instead gives the symmetric
	half-weight appearing in
	Proposition~\ref{prop:Gaussian-Gamma-Stein}.
\end{remark}

\begin{remark}[Variance and polymer geometry]
The diagonal case of Theorem~\ref{thm:multilevel-identity} gives a local
variance identity.  Taking \(m=1\), \(k_1=0\), \(y_1=x\), and
\(\varphi(z)=z\), and using that \(\EE F_x^\be\) is independent of \(x\),
gives
\begin{align}\label{eq:variance-terminal}
	\frac{d}{dx}\Var(F_x^\be)
	=
	2\EE\left[Q_x^\be\{\ga_n>-x\}\right]-1.
\end{align}
Thus the spatial derivative of the free-energy variance is determined by
the annealed distribution of the terminal polymer location.

This is the differential counterpart of the variance identity of
Sepp\"al\"ainen and Valk\'o \cite[Theorem~3.6]{SV10}.  In their stationary
boundary formulation,
\begin{align}
	\Var(\log Z_n^\theta(t))
	=
	n\Psi_1(\theta)-t+2E_{n,t}^\theta(\sigma_0^+),
\end{align}
where \(\Psi_1\) is the trigamma function and \(\sigma_0\) is the boundary
exit point.  The trigamma term comes from the vertical Burke increments:
with \(\theta=\be^{-1}\), Corollary~\ref{cor:Burke} gives \(\Var(\log G_i)=\Psi_1(\be^{-1}).\)
Hence \(n\Psi_1(\be^{-1})\) is the variance accumulated by the vertical
log-Gamma increments, while \(t\) is the horizontal Brownian variance.
This term is independent of the spatial variable and therefore disappears
after differentiation.

Under the reversal between the stationary point-to-line and boundary
formulations, the terminal-location probability in
\eqref{eq:variance-terminal} corresponds to the exit-point probability in
the differentiated Sepp\"al\"ainen--Valk\'o identity.  Thus their variance
formula and \eqref{eq:variance-terminal} express the same fluctuation--geometry
relation at integrated and differential levels, respectively.
\end{remark}

\subsection{The joint Gaussian--Gamma Stein identity}
The multilevel identity enters the Burke argument through the lowest corner
\((x_1,1)\).  The next proposition separates the variables created at this
corner from those carried by the part of the path above it.  This is the
induction step that allows the corners to be removed one at a time.  We begin
by collecting the variables that remain after the lowest corner is removed.

For the path fixed in
Section~\ref{subsec:path-variables}, write
\[
  I_0=I_n=[0,\infty),\qquad
  I_k=[0,x_k-x_{k+1}],\quad 1\leq k<n.
\]
At the corner \((x_1,1)\), fix integers \(p,q\geq0\), with \(q\leq n-1\).
For
\(1\leq r\leq p\), choose \(k_r\in\{1,\ldots,n\}\) and
\(s_r<t_r\) in \(I_{k_r}\), and, for \(1\leq r\leq q\), choose distinct
\(\ell_r\in\{2,\ldots,n\}\).  Set
\begin{align}
  \overrightarrow{\mathcal B}
  &=\bigl(\cB_{k_r}^\be(t_r)-\cB_{k_r}^\be(s_r)\bigr)_{r=1}^p,\notag\\
  \mathbf G^\uparrow&=(G_{\ell_r})_{r=1}^q,\qquad
  U=(\overrightarrow{\mathcal B},\mathbf G^\uparrow)
  \in E=\R^p\times(0,\infty)^q.
  \label{eq:upper-path-variables}
\end{align}
By construction,
\begin{align}\label{eq:upper-environment-measurability}
  \sigma(U)\subset\sigma(B_1,\ldots,B_n).
\end{align}

\begin{proposition}[One-corner Gaussian--Gamma Stein identity]
\label{prop:Gaussian-Gamma-Stein}
Fix \(m\geq0\).  Put \(a_0=x_1\), and, if
\(m\geq1\), choose \(a_0<a_1<\cdots<a_m\).  Set
\[
  X_j=F_{a_j}^\be-F_{a_{j-1}}^\be,\qquad
  \mathbf X=(X_1,\ldots,X_m),\qquad G=G_1.
\]
For every
\(\Phi\in C_c^1(E\times\R^m\times(0,\infty))\),
\begin{align}
  &\EE\left[X_j\Phi(U,\mathbf X,G)\right]
  =(a_j-a_{j-1})
  \EE\left[\partial_{x_j}\Phi(U,\mathbf X,G)\right],
  \qquad 1\leq j\leq m,
  \label{eq:global-Gaussian-Stein}\\*
  &\EE\left[
    G\partial_g\Phi(U,\mathbf X,G)
    +\left(\frac1\be-\frac{G}{\be^2}\right)
    \Phi(U,\mathbf X,G)\right]=0.
  \label{eq:global-Gamma-Stein}
\end{align}
For \(m=0\), only \eqref{eq:global-Gamma-Stein} is asserted.
\end{proposition}

\begin{remark}\label{rem:Stein-operators}
With \(\sigma_j^2=a_j-a_{j-1}\), the operators in
\eqref{eq:global-Gaussian-Stein}--\eqref{eq:global-Gamma-Stein} are
\[
  \Phi\longmapsto\sigma_j^2\partial_{x_j}\Phi-x_j\Phi,\qquad
  \Phi\longmapsto g\partial_g\Phi+(\be^{-1}-\be^{-2}g)\Phi.
\]
They characterize \(N(0,\sigma_{j}^2)\) and
\(\operatorname{Gamma}(\be^{-1},\textup{rate }\be^{-2})\), respectively;
see \cite[Lemma~1]{Mec09} and \cite[Section~2]{DP18}.  
The dependence of \(\Phi\) on \(U\) turns
\eqref{eq:global-Gaussian-Stein}--\eqref{eq:global-Gamma-Stein} into
conditional Stein identities.  In particular,
\[
\mathcal L\left(
X_j\,\middle|\,
\sigma\bigl(U,G,(X_\ell)_{\ell\ne j}\bigr)\right)
=N(0,a_j-a_{j-1}),
\]
and
\[
\mathcal L\left(G\,\middle|\,\sigma(U,\mathbf X)\right)
=
\operatorname{Gamma}
\bigl(\be^{-1},\textup{rate }\be^{-2}\bigr)
\qquad\text{a.s.}
\]
Applying the conditional identities first to smooth compactly supported
product tests and then using a standard approximation and monotone-class
argument gives the corresponding factorization for bounded Borel tests.
Consequently, for bounded Borel functions
\(h:E\to\R\), \(f_j:\R\to\R\), and
\(\psi:(0,\infty)\to\R\),
\begin{align}\label{eq:one-corner-product-law}
	\EE\left[
	h(U)\psi(G)\prod_{j=1}^m f_j(X_j)\right]
	=
	\EE[h(U)]\,\EE[\psi(G)]
	\prod_{j=1}^m\EE[f_j(X_j)].
\end{align}
The product is interpreted as \(1\) when \(m=0\).  Thus
\(X_1,\ldots,X_m,G\) are mutually independent and are independent of
\(U\), with
\[
X_j\sim N(0,a_j-a_{j-1}),\qquad
G\sim\operatorname{Gamma}
\bigl(\be^{-1},\textup{rate }\be^{-2}\bigr).
\]
\end{remark}

\begin{corollary}[Burke property]
\label{cor:Burke}
Fix \(n\geq1\), \(\be>0\), and
\(x_n\leq\cdots\leq x_1\).  The processes
\(\{\cB_k^\be:0\leq k\leq n\}\) and the variables
\(G_i=G_{x_i,i}^{\be}\), \(1\leq i\leq n\), are mutually independent.
Every \(\cB_k^\be\) is a standard Brownian motion on
its parameter interval, and
\begin{align}\label{eq:gamma-law}
  G_i
  \sim
  \operatorname{Gamma}
  \bigl(\be^{-1},\textup{rate }\be^{-2}\bigr),
  \qquad 1\leq i\leq n.
\end{align}
Equivalently, the vertical free-energy increments
\(\cV_{x_i}^{\be}(i)=\be^{-1}\log G_i\) are independent scaled log-gamma
variables and are independent of all
horizontal free-energy increments on the down-right path.
\end{corollary}

\begin{proof}
We prove the finite-dimensional assertion by induction on \(n\).  For
\(n=1\),
\[
\cB_1^\be(s)=B_1(x_1)-B_1(x_1-s),\qquad s\geq0,
\]
is a standard Brownian motion.  For
\(x_1=a_0<a_1<\cdots<a_m\),
\[
X_j
=F_{a_j}^\be-F_{a_{j-1}}^\be
=\cB_0^\be(a_j-x_1)-\cB_0^\be(a_{j-1}-x_1),
\qquad 1\leq j\leq m.
\]
Taking \(U\) to be an arbitrary finite collection of increments of
\(\cB_1^\be\), Proposition~\ref{prop:Gaussian-Gamma-Stein} gives the
finite-dimensional product law of
\(\cB_0^\be,\cB_1^\be\), and \(G_1\), with the stated marginal laws.

Assume that the assertion holds at terminal level \(n-1\).  Applied to
the environment \((B_1,\ldots,B_n)\), the points
\(x_n\leq\cdots\leq x_2\), and the level indices decreased by one, the
induction hypothesis gives the finite-dimensional product law of \(\cB_1^\be,\ldots,\cB_n^\be, G_2,\ldots,G_n\)
on their present parameter intervals.  Let \(U\) be any finite
collection of increments and ratios from this family, and choose
\(x_1=a_0<a_1<\cdots<a_m\).  Proposition
\ref{prop:Gaussian-Gamma-Stein} gives
\[
\EE\left[
h(U)\psi(G_1)\prod_{j=1}^m f_j(X_j)\right]
=
\EE[h(U)]\,\EE[\psi(G_1)]
\prod_{j=1}^m\EE[f_j(X_j)],
\]
where \(X_j=\cB_0^\be(a_j-x_1)-\cB_0^\be(a_{j-1}-x_1).\)
Since \(U\) and \(a_1,\ldots,a_m\) are arbitrary, the
finite-dimensional assertion follows at terminal level \(n\).

For \(k<n\), continuity of \(\cB_k^\be\) follows from
\cite[Lemma~C.2]{RassoulAghaShenZhangZheng26}; for \(k=n\), it follows
from \eqref{eq:top-process}.  Since \(\cB_k^\be(0)=0\), the preceding
finite-dimensional distributions identify every \(\cB_k^\be\) as a
standard Brownian motion on \(I_k\).  If \(D_k\) is a countable dense
subset of \(I_k\), then \(\sigma(\cB_k^\be)=\sigma\{\cB_k^\be(t):t\in D_k\}.\)
The finite-dimensional product law on \(D_0,\ldots,D_n\) therefore gives
the asserted process-level mutual independence.

\end{proof}

\subsection{Zero-temperature limit}
We finish this section by passing it to zero temperature.  The \(\OCY\) polymer converges to
Brownian last-passage percolation (BLPP): the free energy converges to the
last-passage value, while the quenched polymer measure concentrates on
maximizing paths.  We refer to
\cite[Section~2.8]{RassoulAghaShenZhangZheng26} for the corresponding
convergence results.

For \(0\leq k\leq n\), put
\(\overline\Delta_{k,n}=\{0=\ga_k\leq\ga_{k+1}\leq\cdots\leq\ga_n\}\)
and define the stationary point-to-line Brownian last-passage value
\begin{align}\label{eq:BLPP-value}
	&L_{(x,k),n}=\sup_{\ga\in\overline\Delta_{k,n}}
	\left\{B_n(x+\ga_n)+\sum_{j=k}^{n-1}
	\bigl[B_j(x+\ga_{j+1})-B_j(x+\ga_j)\bigr]-\ga_n\right\}.
\end{align}
The sum is empty for \(k=n\), so \(L_{(x,n),n}=B_n(x)\).
Given \(x_n\leq\cdots\leq x_1\), define
\(\cB_n^\infty,\ldots,\cB_0^\infty\) by
\eqref{eq:top-process}--\eqref{eq:bottom-process}, with
\(F_{(x,k),n}^{\be}\) replaced by \(L_{(x,k),n}\), and put
\begin{align}\label{eq:BLPP-vertical}
	&V_i^\infty=L_{(x_i,i),n}-L_{(x_i,i-1),n},
	\qquad 1\leq i\leq n.
\end{align}

\begin{corollary}[Burke property for stationary Brownian LPP]
	\label{cor:BLPP-Burke}
	Fix \(n\geq1\) and \(x_n\leq\cdots\leq x_1\).
	The processes \(\cB_n^\infty,\ldots,\cB_0^\infty\) and the variables
	\(V_1^\infty,\ldots,V_n^\infty\) are mutually independent.  Each
	\(\cB_k^\infty\) is a standard Brownian motion on its parameter interval,
	and \(-V_i^\infty\sim\operatorname{Exp}(1)\), \(1\leq i\leq n\).
\end{corollary}

\begin{proof}
	
Apply \cite[Lemma~2.7]{RassoulAghaShenZhangZheng26} with the stationary case.  Relabeling \(B_k,\ldots,B_n\) as
\(B_0,\ldots,B_{n-k}\) gives, for every compact \(J\subset\R\),
\begin{align}\label{eq:shifted-zero-temperature-limit}
	&\EE\left[\sup_{u\in J}
	\bigl|F_{(u,k),n}^{\be}-L_{(u,k),n}\bigr|^p\right]\longrightarrow0,
	\qquad 0\leq k\leq n,\quad p<\infty.
\end{align}
Fix finite meshes on the horizontal pieces, and let \(X_{k,j}^\be\) and
\(X_{k,j}^\infty\) be the corresponding increments of \(\cB_k^\be\) and
\(\cB_k^\infty\), with mesh lengths \(\ell_{k,j}\).  Then
\eqref{eq:shifted-zero-temperature-limit} gives
\begin{align}\label{eq:finite-mesh-zero-temperature}
	&\left((X_{k,j}^\be)_{k,j},
	(\cV_{x_i}^\be(i))_{i=1}^n\right)
	\xrightarrow[\be\to\infty]{\mathbb P}
	\left((X_{k,j}^\infty)_{k,j},(V_i^\infty)_{i=1}^n\right).
\end{align}
For \(G\sim\operatorname{Gamma}(\be^{-1},\textup{rate }\be^{-2})\), for each \(r\in\R,\)
\begin{align*}
	&\EE[e^{ir\cV}]=\EE[G^{\mathrm ir/\be}]
	=\be^{2\mathrm ir/\be}
	\frac{\Gamma((1+\mathrm ir)/\be)}{\Gamma(1/\be)}
	\longrightarrow\frac1{1+\mathrm ir}
\end{align*}
as \(\be\to \infty,\) where the limit follows from \(z\Gamma(z)\to1\) as \(z\to0\).  Hence
Corollary~\ref{cor:Burke} and \eqref{eq:finite-mesh-zero-temperature} give
\begin{align}\label{eq:BLPP-product-transform}
	&\EE\exp\left\{\mathrm i\sum_{k=0}^n\sum_{j=1}^{m_k}
	t_{k,j}X_{k,j}^\infty+\mathrm i\sum_{i=1}^nr_iV_i^\infty\right\}
	=\prod_{k=0}^n\prod_{j=1}^{m_k}e^{-\ell_{k,j}t_{k,j}^2/2}
	\prod_{i=1}^n\frac1{1+\mathrm ir_i}.
\end{align}
The right-hand side is the product of the characteristic functions of the
stated Gaussian increments and negative rate-one exponential variables.
Moreover, \eqref{eq:shifted-zero-temperature-limit} gives a continuous
version of each \(\cB_k^\infty\): along a subsequence the continuous
free-energy profiles converge locally uniformly almost surely.  Nested finite
meshes with dense union, followed by continuity, therefore upgrade the
finite-dimensional product law to the process statement.
\end{proof}

\begin{remark}[Zero-temperature variance identity]
The preceding Burke property also gives a zero-temperature counterpart of
the variance identity.  Indeed, the diagonal free-energy identity \eqref{eq:variance-terminal} and
the Burke property imply, for \(t\geq0\),
\begin{align}
	\Var(F_{-t}^\be)
	&=
	\frac{n}{\be^2}\Psi_1(\be^{-1})-t
	+2\EE\left[
		\int (t-\ga_n)^+\,Q_0^\be(d\ga)
	\right],
	\label{eq:finite-temp-variance}
\end{align}
where \(\Psi_1\) is the trigamma function.  Here
\(n\be^{-2}\Psi_1(\be^{-1})=\Var(F_0^\be)\) is the variance contributed
by the \(n\) independent vertical log-Gamma Burke increments.

Let \(\Gamma_0(n)\) denote the terminal coordinate of the almost surely
unique maximizing path for \(L_{(0,0),n}\).  Since \(\be^{-2}\Psi_1(\be^{-1})\longrightarrow1\)
and \(Q_0^\be\) concentrates on the maximizing path as
\(\be\to\infty\), passing to the zero-temperature limit in
\eqref{eq:finite-temp-variance} gives
\begin{align}
	\Var\bigl(L_{(-t,0),n}\bigr)
	=
	n-t+2\EE\bigl[(t-\Gamma_0(n))^+\bigr].
	\label{eq:BLPP-variance}
\end{align}
Under the usual reversal between the point-to-line and stationary
boundary formulations, \eqref{eq:BLPP-variance} is the zero-temperature
analogue of the variance--exit identity in \cite{SV10}.
\end{remark}

\section{Proofs of the main results}
\label{sec:proof}

\subsection{Proof strategy}
We first derive the one-corner Proposition~\ref{prop:Gaussian-Gamma-Stein} from the multilevel
Theorem~\ref{thm:multilevel-identity}.  The Gaussian part is the same
prefix-telescoping calculation that leads to
\cite[Equation~(2.47)]{RassoulAghaShenZhangZheng26}, with the multilevel
identity replacing the level-zero formula.  The new point is the Gamma
identity: the diagonal multilevel formula is paired once more with the
level-zero white noise, and the two half-contributions are combined with the
first-jump recursion for the partition function.

We then prove Theorem~\ref{thm:multilevel-identity}.  The argument follows the
mechanism of \cite[Sections~1.3 and 2.1--2.6]{RassoulAghaShenZhangZheng26}.
After mollifying the Brownian environment, Gaussian integration by parts and
the pathwise cancellations of Lemmas~B.1--B.2 there reduce the correlation
identity to a four-point pairing kernel.  The estimates of Proposition~C.3
and Lemmas~C.4--C.5 there justify removal of the mollification, and the
switching identity \cite[Lemma~2.2 and Equations~(2.12)--(2.14)]
{RassoulAghaShenZhangZheng26} identifies the limiting sign kernel.

The new kernel comes from pairing with \(F_{(y,k),n}^{\be}\).  Only levels
\(k,\ldots,n\) are shared with the reference polymer.  Conditional on the
entrance point at level \(k\), the upper part of the reference path has the
tail law \(Q_{(x+\ga_k,k),n}^{\be}\). After relabeling the levels, the
mixed-level switching step is exactly the same identity just cited.  The
remaining issue is the spatial diagonal, where the even mollifier produces a
half weight at level zero.  Both limits are collected in
Lemma~\ref{lem:covariance-kernel-limit}.

\subsection*{Notation used in the proof}
For convenience, we collect the notation used repeatedly in this section.
\begin{description}[leftmargin=2.8cm,style=nextline]
\item[\(\eps\)] A superscript \(\eps\) means that each \(B_i\) is replaced
by its mollification \(B_i^\eps\); superscript \(0\) denotes the original
environment.
\item[\(\phi^\eps,\Theta^\eps,\Psi^\eps\)] The mollifier, its odd primitive,
and the convolved sign approximation, defined in \eqref{eq:mollification}.
\item[\(\xi_i,D_{i,z},\fH\)] The white noise on level \(i\), its Malliavin
derivative at \(z\), and \(\fH=\bigoplus_{i=0}^nL^2(\R)\).
\item[\(\cZ_{(x,k),n}^{\be,\eps}\)] The mollified partition function with
the deterministic factor \(\be^{-(n-k)}\) removed; see
\eqref{eq:mollified-polymer}.
\item[\(\Xi_{i,\eps}\)] The four-point increment in
\eqref{eq:four-point-kernel}.
\item[\(\cK_{0,k}^{\be,\eps}(x,y)\)] The full pairing kernel in
\eqref{eq:covariance-kernel}.
\item[\(\|\cdot\|_{\mathrm{TV}}\)] Total variation with the convention
\(\|\mu-\nu\|_{\mathrm{TV}}=\sup_{\|f\|_\infty\leq1}|\int f\,d\mu-\int f\,d\nu|\).
\end{description}

\subsection{Gaussian integration by parts and Malliavin derivatives}
We record the Malliavin-calculus input used below.  Further properties and
moment estimates for the mollified free energy and its Malliavin derivative
are collected in \cite[Appendix~A]{RassoulAghaShenZhangZheng26}.

Let
\(\xi_i\) be the white-noise derivative of \(B_i\), fix a nonnegative even
\(\phi\in C_c^\infty(\R)\) with \(\int_\R\phi=1\), and set
\begin{align}\label{eq:mollification}
  &\phi^\eps(r)=\eps^{-1}\phi(r/\eps),\qquad
  \xi_i^\eps(x)=\int_\R\phi^\eps(x-z)\,\xi_i(dz),\notag\\[-2pt]
  &B_i^\eps(t)=\int_0^t\xi_i^\eps(s)\,ds,\qquad
  \Theta^\eps(r)=\int_0^r\phi^\eps(s)\,ds,\qquad
  \Psi^\eps=\Theta^\eps*\phi^\eps.
\end{align}
The functions \(\Theta^\eps\) and \(\Psi^\eps\) are odd and
\(\lvert\Psi^\eps\rvert\leq\frac12\).  Fix \(R_\phi\) so that
\(\operatorname{supp}\phi\subset[-R_\phi,R_\phi]\).  Then
\(\Psi^\eps(r)=\frac12\sgn(r)\) for \(|r|>2R_\phi\eps\).  

Write
\(D_{i,z}\) for the Malliavin derivative with respect to \(\xi_i\) at
\(z\), on the isonormal space \(\fH=\bigoplus_{i=0}^nL^2(\R)\); see
\cite[Sections~1.1--1.3]{Nua06}.  The form of Gaussian integration by parts
used below is the standard duality relation; see \cite[Section~1.3]{Nua06}.
\begin{lemma}[Gaussian integration by parts]\label{lem:Gaussian-IBP}
Let \(W=\{W(h):h\in\fH\}\) be the isonormal Gaussian process.  For every
\(F\in\DD^{1,2}\) and \(h\in\fH\),
\begin{align}\label{eq:Gaussian-IBP}
\EE[F W(h)]
=\EE\bigl[\langle DF,h\rangle_{\fH}\bigr]
=\sum_{i=0}^n\int_\R\EE[D_{i,y}F]h_i(y)\,dy.
\end{align}
\end{lemma}
Thus a factor linear in the Gaussian environment can be transferred to the
other random factors in the expectation at the cost of a Malliavin derivative.

With the analytic notation fixed, we first derive the one-corner identity
from the multilevel theorem.  The following subsections then prove the theorem
itself.

\subsection{The Gaussian--Gamma Stein identity from the multilevel identity}
\begin{proof}[Proof of Proposition~\ref{prop:Gaussian-Gamma-Stein}]

Write \(x=x_1\) and first suppose that \(m\geq1\). List the distinct
endpoint free energies occurring in \((U,\mathbf X,G)\), in their order
along the path, as
\[
V_r=F_{(y_r,k_r),n}^{\be},\qquad 0\leq r\leq N.
\]
Choose \(\ell\) so that
\begin{align}\label{eq:corner-vertices}
	V_{\ell-1}=F_{(x,1),n}^{\be},\qquad
	V_\ell=F_x^\be,\qquad
	V_{\ell+j}=F_{a_j}^\be,\quad 1\leq j\leq m.
\end{align}
Regard
\[
\widehat\Phi(V_0,\ldots,V_N)=\Phi(U,\mathbf X,G)
\]
as a function of the vertex values, and set
\[
\Phi_\eps=\Phi(U^\eps,\mathbf X^\eps,G^\eps),\qquad
(\partial_r\widehat\Phi)^\eps
=(\partial_r\widehat\Phi)(V_0^\eps,\ldots,V_N^\eps).
\]
Since all arguments of \(\Phi\) are functions of differences of vertex
values, a common shift gives
\begin{align}\label{eq:total-derivative-zero}
	\sum_{r=0}^N(\partial_r\widehat\Phi)^\eps=0.
\end{align}
Similarly, shifting \(V_0^\eps,\ldots,V_{\ell+j-1}^\eps\) by \(t\)
changes only \(X_j^\eps\), replacing it by \(X_j^\eps-t\). Hence
\begin{align}\label{eq:horizontal-telescoping}
	-\sum_{r=0}^{\ell+j-1}(\partial_r\widehat\Phi)^\eps
	=
	\partial_{x_j}\Phi_\eps,
	\qquad 1\leq j\leq m.
\end{align}

We record more explicitly the two shifts at the corner, since they will
also produce the Gamma identity. By \eqref{eq:corner-vertices},
\[
G^\eps=e^{\be(V_{\ell-1}^\eps-V_\ell^\eps)},
\qquad
X_1^\eps=V_{\ell+1}^\eps-V_\ell^\eps.
\]
Shifting \(V_0^\eps,\ldots,V_{\ell-1}^\eps\) by \(t\) leaves
\(U^\eps\) and \(\mathbf X^\eps\) unchanged and sends
\(G^\eps\) to \(e^{\be t}G^\eps\). Therefore
\begin{align}\label{eq:vertical-telescoping}
	\sum_{r=0}^{\ell-1}(\partial_r\widehat\Phi)^\eps
	=
	\be G^\eps\partial_g\Phi_\eps.
\end{align}
On the other hand, shifting \(V_0^\eps,\ldots,V_\ell^\eps\) by \(t\)
leaves \(U^\eps\) and \(G^\eps\) unchanged and sends
\(X_1^\eps\) to \(X_1^\eps-t\). Thus
\[
-\sum_{r=0}^{\ell}(\partial_r\widehat\Phi)^\eps
=
\partial_{x_1}\Phi_\eps.
\]
Combining this with \eqref{eq:vertical-telescoping} gives
\begin{align}
	\frac12(\partial_\ell\widehat\Phi)^\eps
	&=
	-\frac{\be}{2}G^\eps\partial_g\Phi_\eps
	-\frac12\partial_{x_1}\Phi_\eps,
	\label{eq:corner-derivative}\\
	-\sum_{r=0}^{\ell-1}(\partial_r\widehat\Phi)^\eps
	-\frac12(\partial_\ell\widehat\Phi)^\eps
	&=
	-\frac{\be}{2}G^\eps\partial_g\Phi_\eps
	+\frac12\partial_{x_1}\Phi_\eps.
	\label{eq:corner-telescoping}
\end{align}
The same identities hold at \(\eps=0\).

The Gaussian identity \eqref{eq:global-Gaussian-Stein} is the same
telescoping calculation as
\cite[Equation~(2.47)]{RassoulAghaShenZhangZheng26}, with the upper-path
variables \(U\) retained in the test function. Equivalently, applying
Theorem~\ref{thm:multilevel-identity}(i) for
\(s\in(a_{j-1},a_j)\), using
\eqref{eq:total-derivative-zero} and
\eqref{eq:horizontal-telescoping}, gives
\[
\frac{\partial}{\partial s}
\EE\left[F_s^\be\Phi(U,\mathbf X,G)\right]
=
\EE\left[\partial_{x_j}\Phi(U,\mathbf X,G)\right].
\]
Integration over \((a_{j-1},a_j)\) yields
\eqref{eq:global-Gaussian-Stein}.

It remains to prove the Gamma identity. At \(s=x\), the entrance coefficients in
\eqref{eq:multilevel-diagonal} are
\[
Q_x^\be\{\ga_{\kappa_r}>v_r-x\}
+\frac12\ind_{\{v_r=x,\,\kappa_r=0\}}
=\begin{cases}1,&r<\ell,\\[2pt]\frac12,&r=\ell,\\[2pt]0,&r>\ell.
\end{cases}
\]
Thus \eqref{eq:corner-telescoping} gives
\begin{align}\label{eq:corner-diagonal-identity}
\lim_{\eps\downarrow0}\EE[\Phi_\eps\partial_xF_x^{\be,\eps}]
=-\frac\be2\EE[G\partial_g\Phi]
+\frac12\EE[\partial_{x_1}\Phi],
\end{align}
where the arguments \((U,\mathbf X,G)\) are suppressed on the right.

Next, we pair the same test function with the level-zero white noise.  Since
\(U^\eps\) depends only on \(B_1^\eps,\ldots,B_n^\eps\),
\(D_{0,z}U^\eps=0\).  For \(r\geq\ell\),
\[
D_{0,z}F_{v_r}^{\be,\eps}
=\int[\Theta^\eps(v_r+\ga_1-z)-\Theta^\eps(v_r-z)]
Q_{v_r}^{\be,\eps}(d\ga).
\]
The Malliavin chain rule \cite[Proposition~1.2.4]{Nua06} and
Lemma~\ref{lem:Gaussian-IBP} give
\begin{align}\label{eq:white-noise-expansion}
\EE[\xi_0^\eps(x)\Phi_\eps]
=\sum_{r=\ell}^N\EE[(\partial_r\widehat\Phi)^\eps J_r^\eps],
\end{align}
with
\begin{align}\label{eq:white-noise-kernel}
J_r^\eps=\int[\Psi^\eps(v_r+\ga_1-x)-\Psi^\eps(v_r-x)]
Q_{v_r}^{\be,\eps}(d\ga).
\end{align}
Because \(v_\ell=x\),
\begin{align}
	\left|J_\ell^\eps-\frac12\right|
	&\leq
	\left|
	\int \Psi^\eps(\ga_1)
	\bigl(Q_x^{\be,\eps}-Q_x^\be\bigr)(d\ga)
	\right|
	\notag
	+
	\left|
	\int
	\left(\Psi^\eps(\ga_1)-\frac12\right)
	Q_x^\be(d\ga)
	\right|
	\notag\\
	&\leq
	\|Q_x^{\be,\eps}-Q_x^\be\|_{\mathrm{TV}}
	+
	\int
	\left|\Psi^\eps(\ga_1)-\frac12\right|
	Q_x^\be(d\ga)
	\notag\\
	&\leq
	\|Q_x^{\be,\eps}-Q_x^\be\|_{\mathrm{TV}}
	+
	Q_x^\be\{\ga_1\leq 2R_\phi\eps\}
	\longrightarrow0.
	\label{eq:J-diagonal-limit}
\end{align}
For \(r>\ell\), the path ordering gives \(v_r>x\), and hence
\(J_r^\eps=0\) once \(2R_\phi\eps<v_r-x\).  The convergence estimates in
\cite[Proposition~C.3 and Lemma~A.3(iv)]{RassoulAghaShenZhangZheng26}
therefore justify passage to the limit in
\eqref{eq:white-noise-expansion}.  Using \eqref{eq:corner-derivative},
\begin{align}\label{eq:white-noise-identity}
\lim_{\eps\downarrow0}\EE[\xi_0^\eps(x)\Phi_\eps]
=-\frac\be2\EE[G\partial_g\Phi]-\frac12\EE[\partial_{x_1}\Phi].
\end{align}
Adding this to \eqref{eq:corner-diagonal-identity} cancels the
\(\partial_{x_1}\Phi\) terms and gives
\begin{align}\label{eq:combined-contraction}
\lim_{\eps\downarrow0}\EE[\Phi_\eps(\partial_xF_x^{\be,\eps}
+\xi_0^\eps(x))]
=-\be\,\EE[G\partial_g\Phi(U,\mathbf X,G)].
\end{align}

Finally, we identify the random factor in
\eqref{eq:combined-contraction} through the first-jump decomposition.
Set \(s=\ga_1\), \(t=x+s\), and
\(\eta_j=\ga_j-\ga_1\), \(1\leq j\leq n\). Then
\[
H_x^\eps(\ga)
=
B_0^\eps(t)-B_0^\eps(x)
+
H_{(t,1),n}^\eps(\eta),
\]
while the reference measure factors at the first jump.  Integrating first
over \(s\) gives
\begin{align}\label{eq:first-jump-tail}
	Z_x^{\be,\eps}
	&=
	e^{-\be B_0^\eps(x)+\be x}
	\int_x^\infty
	e^{\be[B_0^\eps(t)-t]}
	Z_{(t,1),n}^{\be,\eps}\,dt.
\end{align}
Write \(I^\eps(x)=\int_x^\infty e^{\be[B_0^\eps(t)-t]} Z_{(t,1),n}^{\be,\eps}\,dt,\)
so that \(F_x^{\be,\eps}=-B_0^\eps(x)+x+\frac1\be\log I^\eps(x).\)
By differentiation with respect to \(x\),
\[
(I^\eps)'(x)
=
-e^{\be[B_0^\eps(x)-x]}
Z_{(x,1),n}^{\be,\eps}.
\]
Consequently,
\begin{align}
	\partial_xF_x^{\be,\eps}
	=
	-\xi_0^\eps(x)+1
	-\frac1\be
	\frac{
		e^{\be[B_0^\eps(x)-x]}Z_{(x,1),n}^{\be,\eps}
	}{
		I^\eps(x)
	}
    =
	-\xi_0^\eps(x)+1
	-\frac1\be
	\frac{Z_{(x,1),n}^{\be,\eps}}
	{Z_x^{\be,\eps}},
\end{align}
where the second equality follows from
\eqref{eq:first-jump-tail}. Hence
\begin{align}\label{eq:first-jump-derivative}
	\partial_xF_x^{\be,\eps}+\xi_0^\eps(x)
	&=
	1-\frac1\be
	\frac{Z_{(x,1),n}^{\be,\eps}}
	{Z_x^{\be,\eps}}
	=
	1-\frac1\be G^\eps.
\end{align}

By \cite[Lemma~A.3(iv)]{RassoulAghaShenZhangZheng26} and compact support of
\(\Phi\), the terms in \eqref{eq:combined-contraction} converge in
\(L^1\).  Substitution gives
\[
\EE[\Phi]-\be^{-1}\EE[G\Phi]
=-\be\,\EE[G\partial_g\Phi],
\]
which is equivalent to \eqref{eq:global-Gamma-Stein}.

For \(m=0\), fix \(a_1>x\), choose \(\chi\in C_c^1(\R)\) with \(\chi=1\)
near zero, and apply the \(m=1\) identity to
\(\Phi_R(u,z,g)=\Phi(u,g)\chi(z/R)\).  Letting \(R\to\infty\) and using
dominated convergence proves the claim.
\end{proof}

\subsection{The mollified integration-by-parts formula}

The one-corner proposition used the multilevel identity only through its
statement.  We now prove that identity.  The first step is an exact formula in
the mollified environment. All limiting arguments are postponed to the next
subsection.

The Malliavin derivative of a test free energy records the mollified indicator
of each of its horizontal segments.  If that polymer starts on level
\(k\), only the common levels \(k,\ldots,n\) are paired with the reference
polymer; each pairing gives one four-point increment.  The terms in which
the derivative hits the normalized Gibbs density cancel pathwise by
\cite[Lemmas~B.1 and B.2]{RassoulAghaShenZhangZheng26}.

For the integration-by-parts calculation it is convenient to suppress the
deterministic normalization of the partition function.  For
\(0\leq k\leq n\), write
\begin{align}\label{eq:mollified-polymer}
	\begin{aligned}
		&\cZ_{(x,k),n}^{\be,\eps}=
		\int_{\Delta_{k,n}}
		e^{\be H_{(x,k),n}^\eps(\ga)}
		\mu_\be^{k,n}(d\ga)
		=\be^{n-k}Z_{(x,k),n}^{\be,\eps},
		\\
		&Q_{(x,k),n}^{\be,\eps}(d\ga)=
		\frac{e^{\be H_{(x,k),n}^\eps(\ga)}}
		{\cZ_{(x,k),n}^{\be,\eps}}
		\mu_\be^{k,n}(d\ga).
	\end{aligned}
\end{align}
The deterministic factor does not affect spatial or Malliavin
derivatives.  In the level-zero case we use the stationary
abbreviations
\(\cZ_x^{\be,\eps}=\cZ_{(x,0),n}^{\be,\eps}\),
\(Q_x^{\be,\eps}=Q_{(x,0),n}^{\be,\eps}\), and
\(F_x^{\be,\eps}=F_{(x,0),n}^{\be,\eps}\).

We next introduce the kernel produced by pairing two horizontal energy
increments.  For \(\ga\in\DN\) and \(\eta\in\Delta_{k,n}\), put
\(u_i=x+\ga_i\) and \(v_i=y+\eta_i\).  For
\(k\leq i\leq n-1\), set
\begin{align}\label{eq:four-point-kernel}
	\Xi_{i,\eps}^{(y,k),(x,0)}(\eta,\ga)
	=\Psi^\eps(v_{i+1}-u_{i+1})-\Psi^\eps(v_i-u_{i+1})-\Psi^\eps(v_{i+1}-u_i)+\Psi^\eps(v_i-u_i).
\end{align}
The four terms come from the two endpoints of each segment. Together they
form the rectangular increment of \(\Psi^\eps\).  After summing over the
common levels and adding the contribution from the terminal Brownian motion,
we obtain
\begin{align}\label{eq:covariance-kernel}
	\cK_{0,k}^{\be,\eps}(x,y)=
	\iint
	\Bigg\{
	\sum_{i=k}^{n-1}
	\Xi_{i,\eps}^{(y,k),(x,0)}(\eta,\ga)
	+\Psi^\eps(y+\eta_n-x-\ga_n)
	+\Psi^\eps(x+\ga_n)
	\Bigg\}\times
	Q_{(y,k),n}^{\be,\eps}(d\eta)
	Q_x^{\be,\eps}(d\ga),
\end{align}
where the sum is empty when \(k=n\).

\begin{lemma}[Mollified integration-by-parts identity]
	\label{lem:mollified-ibp}
	Fix \(m\geq1\), \(\mathbf y=(y_1,\ldots,y_m)\in\R^m\), and
	\(\mathbf k=(k_1,\ldots,k_m)\in\{0,\ldots,n\}^m\).  Let
	\(\mathbf F^\eps=
	(F_{(y_1,k_1),n}^{\be,\eps},\ldots,
	F_{(y_m,k_m),n}^{\be,\eps})\), and suppose that
	\(\varphi\in C^1(\R^m)\) and \(\|\nabla\varphi\|_\infty<\infty\).
	Then, for \(\eps>0\),
	\begin{align}\label{eq:mollified-ibp}
		\frac{\partial}{\partial x}
		\EE\left[
		\varphi(\mathbf F^\eps)F_x^{\be,\eps}
		\right]
		=
		\sum_{j=1}^{m}
		\EE\left[
		\partial_j\varphi(\mathbf F^\eps)
		\cK_{0,k_j}^{\be,\eps}(x,y_j)
		\right].
	\end{align}
\end{lemma}

\begin{proof}
	We give the calculation first.  Differentiation under the path integral and
	the applications of Fubini below are justified by the endpoint and
	Malliavin bounds in
	\cite[Proposition~A.5(i) and (iii)]{RassoulAghaShenZhangZheng26}.  The same
	bounds hold for a starting level \(k\) after shifting the level indices.
	The assumption on
	\(\nabla\varphi\) also implies that \(\varphi\) has at most linear growth.
	Differentiating the level-zero partition function gives
	\begin{align}\label{eq:mollified-spatial-derivative}
		&\partial_xF_x^{\be,\eps}=
		\int
		\Bigg[
		\xi_n^\eps(x+\ga_n)
		+\sum_{i=0}^{n-1}
		\bigl(
		\xi_i^\eps(x+\ga_{i+1})
		-\xi_i^\eps(x+\ga_i)
		\bigr)
		\Bigg]
		Q_x^{\be,\eps}(d\ga).
	\end{align}
	
	For the free energy started from \((y,k)\), differentiation with respect to
	the white noise on level \(i\) gives
	\begin{align}
		&D_{i,z}F_{(y,k),n}^{\be,\eps}=
		\int
		\bigl[
		\Theta^\eps(y+\eta_{i+1}-z)
		-\Theta^\eps(y+\eta_i-z)
		\bigr]
		Q_{(y,k),n}^{\be,\eps}(d\eta),
		\qquad k\leq i\leq n-1,
		\label{eq:bulk-Malliavin-derivative}\\
		&D_{n,z}F_{(y,k),n}^{\be,\eps}=
		\int
		\bigl[
		\Theta^\eps(y+\eta_n-z)-\Theta^\eps(-z)
		\bigr]
		Q_{(y,k),n}^{\be,\eps}(d\eta),
		\label{eq:terminal-Malliavin-derivative}
	\end{align}
	and \(D_{i,z}F_{(y,k),n}^{\be,\eps}=0\) for \(i<k\).  Indeed, the
	Malliavin derivative records the mollified indicator of the horizontal
	segment on level \(i\). The factor \(\be\) from the Gibbs weight cancels the
	factor \(\be^{-1}\) in the free energy.
	
	By \cite[Proposition~A.5(iii)]{RassoulAghaShenZhangZheng26}, every component
	of \(\mathbf F^\eps\) belongs to \(\DD^{1,2}\).  The Malliavin chain rule \cite[Proposition~1.2.4]{Nua06}
	therefore gives
	\begin{align}\label{eq:multilevel-chain-rule}
		D\varphi(\mathbf F^\eps)
		=
		\sum_{j=1}^{m}
		\partial_j\varphi(\mathbf F^\eps)
		DF_{(y_j,k_j),n}^{\be,\eps}.
	\end{align}

	Apply Lemma~\ref{lem:Gaussian-IBP}, in the form
	\eqref{eq:Gaussian-IBP}, to each white-noise term in
	\eqref{eq:mollified-spatial-derivative}. When the Malliavin derivative
	falls on \(\varphi(\mathbf F^\eps)\), the chain rule produces the desired
	pairing with the test free energies. The remaining terms, in which the
	derivative falls on the normalized Gibbs density, vanish by the
	cancellation identities
	\cite[Lemmas~B.1 and B.2]{RassoulAghaShenZhangZheng26}. More precisely,
	with \(\rho_x^\eps(\ga)=\frac{e^{\be H_x^\eps(\ga)}}{\cZ_x^{\be,\eps}},\)
	these identities give, almost surely,
	\begin{align}
		\int_\R dz\int\mu_\be(d\ga)\,
		D_{i,z}\rho_x^\eps(\ga)
		\bigl[
		\phi^\eps(x+\ga_{i+1}-z)
		-\phi^\eps(x+\ga_i-z)
		\bigr]
		&=0,
		\label{eq:bulk-cancellation}\\
		\int_\R dz\int\mu_\be(d\ga)\,
		D_{n,z}\rho_x^\eps(\ga)
		\phi^\eps(x+\ga_n-z)
		&=0.
		\label{eq:terminal-cancellation}
	\end{align}
	The cancellations are pathwise and therefore remain valid in the present
	calculation after multiplication by the factors arising from
	\(\varphi(\mathbf F^\eps)\). Thus only the terms in which the Malliavin
	derivative acts on the test free energies remain.
	
	It remains to collect the terms in which the derivative hits the test function.
	Insert
	\eqref{eq:multilevel-chain-rule} and
	\eqref{eq:bulk-Malliavin-derivative} into a bulk term of
	\eqref{eq:mollified-spatial-derivative}.  Convolution in the
	Malliavin variable \(z\) gives
	\begin{align*}
		\int_\R
		\bigl[
		\Theta^\eps(y_j+\eta_{i+1}-z)
		-\Theta^\eps(y_j+\eta_i-z)
		\bigr]\times
		\bigl[
		\phi^\eps(x+\ga_{i+1}-z)
		-\phi^\eps(x+\ga_i-z)
		\bigr]\,dz,
	\end{align*}
	which is \(	\Xi_{i,\eps}^{(y_j,k_j),(x,0)}(\eta,\ga).\)
	The terminal derivative \eqref{eq:terminal-Malliavin-derivative} is computed
	in the same way and gives
	\begin{align*}
		\int_\R
		\bigl[
		\Theta^\eps(y_j+\eta_n-z)-\Theta^\eps(-z)
		\bigr]
		\phi^\eps(x+\ga_n-z)\,dz=
		\Psi^\eps(y_j+\eta_n-x-\ga_n)
		+\Psi^\eps(x+\ga_n).
	\end{align*}
	Integrating against the two quenched polymer measures gives the terms in
	\(\cK_{0,k_j}^{\be,\eps}(x,y_j)\).  Summing over \(j\) gives
	\eqref{eq:mollified-ibp}.
\end{proof}

\subsection{Limit of the pairing kernel}

We follow the removal-of-mollification step in
\cite[Section~1.3]{RassoulAghaShenZhangZheng26}.  It is useful to separate
the argument into two parts.  First, total-variation convergence and a
small-diagonal estimate replace the mollified sign function by its limit.
Second, the switching identity identifies the resulting sign kernel with an
entrance probability under the quenched measure.  The same argument applies on the diagonal, where the
symmetry of the mollifier produces the factor \(1/2\).

\begin{lemma}[Limit of the pairing kernel]
	\label{lem:covariance-kernel-limit}
	Fix \(0\leq k\leq n\).
	
	\noindent\textup{(i)} If \(y\in\R\) and
	\(I\Subset\R\setminus\{y\}\), then, almost surely and uniformly for
	\(x\in I\),
	\begin{align}\label{eq:off-diagonal-kernel-limit}
		\cK_{0,k}^{\be,\eps}(x,y)
		\longrightarrow
		Q_x^\be\{\ga_n>-x\}
		-Q_x^\be\{\ga_k>y-x\}.
	\end{align}
	
	\noindent\textup{(ii)} For every compact \(J\subset\R\), almost surely
	and uniformly for \(x\in J\),
	\begin{align}\label{eq:diagonal-kernel-limit}
		\cK_{0,k}^{\be,\eps}(x,x)
		\longrightarrow
		Q_x^\be\{\ga_n>-x\}
		-\ind_{\{k>0\}}-\frac12\ind_{\{k=0\}}.
	\end{align}
	In both cases, the supremum over the indicated compact set of the
	absolute difference also converges to zero in \(L^p(\Omega)\) for every
	finite \(p\).
\end{lemma}

\begin{proof}
	For every compact \(J\subset\R\), Proposition~C.3 of
	\cite{RassoulAghaShenZhangZheng26}, applied after relabeling
	\(B_k,\ldots,B_n\) as \(B_0,\ldots,B_{n-k}\), gives the locally uniform
	total-variation convergence
	\begin{align}\label{eq:quenched-TV-convergence}
		\sup_{x\in J}
		\left\|
		Q_{(x,k),n}^{\be,\eps}
		-Q_{(x,k),n}^{\be}
		\right\|_{\mathrm{TV}}
		\longrightarrow0.
	\end{align}
	The corresponding product measures therefore converge locally uniformly in
	total variation as well.
	
	Let \(R_\phi\) satisfy \(\operatorname{supp}\phi\subset[-R_\phi,R_\phi]\).
	The only obstruction to replacing \(\Psi^\eps\) by \(\frac12\sgn\) is a
	small neighborhood of zero.  Under the reference measures every nonconstant
	coordinate difference \(R(\eta,\ga)=y+\eta_p-x-\ga_q\) has a locally
	bounded density, uniformly for \(x\) in a compact set.  Consequently its
	reference probability of \(\{|R|\leq\delta\}\) is \(O(\delta)\).
	Cauchy--Schwarz and the uniform \(L^2\) bounds for the Gibbs densities in
	\cite[Proposition~C.1 and Lemma~C.2]{RassoulAghaShenZhangZheng26}, after
	relabeling the upper levels when \(k>0\), give
	\begin{align}\label{eq:small-diagonal}
		\lim_{\delta\downarrow0}
		\sup_{\substack{x\in I\\0\leq\eps\leq\eps_I}}
		(Q_{(y,k),n}^{\be,\eps}\otimes Q_x^{\be,\eps})
		\{|R(\eta,\ga)|\leq\delta\}=0
	\end{align}
	for some \(\eps_I>0\).  On the spatial diagonal the same estimate holds
	uniformly for \(x\) in the compact set \(J\) from part~(ii), for every
	nonconstant difference.  An identically zero difference contributes nothing
	because \(\Psi^\eps(0)=\sgn(0)=0\).

	Put
	\(\nu_{x,y}^\eps=Q_{(y,k),n}^{\be,\eps}\otimes
	Q_x^{\be,\eps}\) and
	\(\nu_{x,y}=Q_{(y,k),n}^{\be}\otimes Q_x^\be\).  Then
	\begin{align}
		\left|
		\int\Psi^\eps(R)\,d\nu_{x,y}^\eps
		-\frac12\int\sgn(R)\,d\nu_{x,y}
		\right|\leq
		\|\nu_{x,y}^\eps-\nu_{x,y}\|_{\mathrm{TV}}
		+
		\nu_{x,y}\{|R|\leq2R_\phi\eps\}\longrightarrow 0.
		\label{eq:sign-kernel-convergence}
	\end{align}
	Combining \eqref{eq:sign-kernel-convergence} with
	\eqref{eq:quenched-TV-convergence}, term by term in
	\eqref{eq:covariance-kernel}, gives locally uniform convergence to a sign
	kernel.  This separates the analytic part of the limit from the geometric
	identification that follows.
	
	Denote the limiting kernel by \(\cK_{0,k}^{\be}(x,y)\).  Suppose first
	that \(k>0\).  Condition the reference polymer on
	\((\ga_1,\ldots,\ga_k)\).  By the memoryless property of the Poisson
	reference measure and factorization of the Gibbs weight, the shifted upper
	tail has quenched law \(Q_{(x+\ga_k,k),n}^{\be}\).  After relabeling levels
	\(k,\ldots,n\) as \(0,\ldots,n-k\), the two tails are in the setting of
	\cite[Lemma~2.2, Equation~(2.14)]{RassoulAghaShenZhangZheng26}.  The
	rectangular sign increments and the terminal cross term collapse to the sign
	of the separation at level \(k\).  Averaging over the lower segment and using \(\frac12\sgn(u)=\ind_{\{u>0\}}-\frac12,\)
    gives
	\begin{align}
		\cK_{0,k}^{\be}(x,y)
		&=
		\frac12\int\sgn(y-x-\ga_k)\,Q_x^\be(d\ga)
		+\frac12\int\sgn(x+\ga_n)\,Q_x^\be(d\ga)
		\notag\\
		&=
		Q_x^\be\{\ga_n>-x\}
		-
		Q_x^\be\{\ga_k>y-x\}.
		\label{eq:mixed-sign-identification}
	\end{align}
	If \(k=0\) and \(x\ne y\), the same identity reduces to
	\begin{align}
		&\cK_{0,0}^{\be}(x,y)=
		\frac12\sgn(y-x)
		+
		\frac12
		\int\sgn(x+\ga_n)
		Q_x^\be(d\ga)\notag\\*
		&\hspace{10mm}=
		Q_x^\be\{\ga_n>-x\}
		-\ind_{\{x>y\}}.
		\label{eq:same-level-sign-identification}
	\end{align}
	The quenched measures are absolutely continuous with respect to the
	Poisson jump-time law, so their endpoint laws have no atoms and no
	correction at zero is needed.  Equations
	\eqref{eq:mixed-sign-identification} and
	\eqref{eq:same-level-sign-identification} prove
	\eqref{eq:off-diagonal-kernel-limit}.
	
	On the diagonal, \(\ga_k>0\) almost surely when \(k>0\), and
	\eqref{eq:mixed-sign-identification} gives
	\(\cK_{0,k}^{\be}(x,x)=Q_x^\be\{\ga_n>-x\}-1\).
	When \(k=0\), interchange of the two replicas changes the sign of
	every term in \eqref{eq:covariance-kernel} except the anchored terminal
	term.  Consequently,
	\begin{align*}
		\cK_{0,0}^{\be,\eps}(x,x)=
		\int\Psi^\eps(x+\ga_n)
		Q_x^{\be,\eps}(d\ga)\longrightarrow
		\frac12
		\int\sgn(x+\ga_n)
		Q_x^\be(d\ga)
		=
		Q_x^\be\{\ga_n>-x\}-\frac12.
	\end{align*}
	This proves \eqref{eq:diagonal-kernel-limit}.  Since the kernels are bounded
	by a deterministic constant depending only on \(n\), the \(L^p\)
	convergence follows by dominated convergence.
\end{proof}

\subsection{Proof of the multilevel identity}

The preceding lemma contains the geometric part of the argument.  We now pass
from the mollified identity to its limit, first away from the marked points
and then across positive-level coincidences.  The level-zero diagonal is
handled separately by the even mollification.

\begin{proof}[Proof of Theorem~\ref{thm:multilevel-identity}]
	We first work away from all the points \(y_j\); the extension across a point
	whose starting level is positive will be handled at the end.  Write
	\(\mathbf F=(F_{(y_1,k_1),n}^{\be},\ldots,
	F_{(y_m,k_m),n}^{\be})\) and
	\(\mathbf F^\eps=(F_{(y_1,k_1),n}^{\be,\eps},\ldots,
	F_{(y_m,k_m),n}^{\be,\eps})\).  Set
	\(A_\eps(x)=\EE[\varphi(\mathbf F^\eps)F_x^{\be,\eps}]\) and
	\(A(x)=\EE[\varphi(\mathbf F)F_x^\be]\), and fix a compact interval
	\(I\Subset\R\setminus\{y_1,\ldots,y_m\}\).
	By Lemma~\ref{lem:mollified-ibp},
	\begin{align}\label{eq:mollified-ibp-on-I}
		A_\eps'(x)
		=
		\sum_{j=1}^{m}
		\EE\left[
		\partial_j\varphi(\mathbf F^\eps)
		\cK_{0,k_j}^{\be,\eps}(x,y_j)
		\right],
		\qquad x\in I.
	\end{align}
	
	By \cite[Lemma~A.3(iv)]{RassoulAghaShenZhangZheng26}, applied after
	relabeling the levels for each starting height \(k_j\),
	\(\mathbf F^\eps\to\mathbf F\) in \(L^p(\Omega)\) for every finite
	\(p\).  Since \(\partial_j\varphi\) is bounded and continuous,
	\(\partial_j\varphi(\mathbf F^\eps)\to
	\partial_j\varphi(\mathbf F)\) in \(L^1(\Omega)\).
	Write \(\cK_{0,k_j}^{\be}(x,y_j)\) for the limit in
	\eqref{eq:off-diagonal-kernel-limit}, and let \(C\) be a common deterministic
	bound for these kernels.  Add and subtract the mixed term with
	\(\partial_j\varphi(\mathbf F)\) and
	\(\cK_{0,k_j}^{\be,\eps}\).  The \(L^1\) convergence of the first factor,
	the boundedness of \(\partial_j\varphi\), and
	Lemma~\ref{lem:covariance-kernel-limit} then give uniform convergence of the
	derivatives on \(I\):
	\begin{align}\label{eq:uniform-derivative-limit}
		A_\eps'(x)
		\longrightarrow
		\sum_{j=1}^{m}
		\EE\left[
		\partial_j\varphi(\mathbf F)
		\left\{
		Q_x^\be\{\ga_n>-x\}
		-Q_x^\be\{\ga_{k_j}>y_j-x\}
		\right\}
		\right].
	\end{align}
	
	The moment bounds in
	\cite[Lemma~A.3(iii)]{RassoulAghaShenZhangZheng26}, with the same level
	relabeling, together with the linear growth of \(\varphi\) and
	H\"older's inequality, also give
	\(A_\eps(x)\to A(x)\) at each fixed \(x\).  The uniform convergence theorem
	for derivatives now applies; hence \(A\) is continuously differentiable on
	\(I\), with derivative given by
	\eqref{eq:uniform-derivative-limit}.
	
	It remains to decide which of the points \(y_j\) are genuine singularities.
	Let \(q_x^\be=dQ_x^\be/d\mu_\be\).  If \(x_r\to x\), choose a compact
	interval containing the sequence and its limit.  The Brownian growth bound
	used in \cite[Proposition~C.1]{RassoulAghaShenZhangZheng26} gives an
	integrable majorant for the corresponding Gibbs densities on this compact
	set, while the Hamiltonian and partition function are continuous in the base
	point.  Dominated convergence therefore gives
	\begin{align}\label{eq:spatial-TV-continuity}
		&\|Q_{x_r}^\be-Q_x^\be\|_{\mathrm{TV}}
		=\int_{\DN}|q_{x_r}^\be-q_x^\be|\,d\mu_\be\longrightarrow0.
	\end{align}
	If \(k_j>0\) and \(x_r\to y_j\), then
	\begin{align*}
		&\left|Q_{x_r}^\be\{\ga_{k_j}>y_j-x_r\}-1\right|
		\leq\|Q_{x_r}^\be-Q_{y_j}^\be\|_{\mathrm{TV}}
		+Q_{y_j}^\be\{0<\ga_{k_j}\leq|x_r-y_j|\}\longrightarrow0,\\
		&\left|Q_{x_r}^\be\{\ga_n>-x_r\}
		-Q_{y_j}^\be\{\ga_n>-y_j\}\right|
		\leq\|Q_{x_r}^\be-Q_{y_j}^\be\|_{\mathrm{TV}}
		+Q_{y_j}^\be\{|\ga_n+y_j|\leq|x_r-y_j|\}\longrightarrow0.
	\end{align*}
	Here both final probabilities vanish because the quenched coordinate laws
	have no atoms.
	The right-hand side of \eqref{eq:multilevel-identity} consequently has the
	same limit from both sides.  Integrating the identity up to \(y_j\) from
	either side shows that \(A\) is differentiable there with this common
	derivative.  Hence
	\eqref{eq:multilevel-identity} holds, with a continuous derivative, on
	every connected component of
	\(\R\setminus\{y_j:k_j=0\}\).
	
	For the final, diagonal statement, fix \(x\in\R\) and return to the
	mollified identity \eqref{eq:mollified-ibp}, which has no off-diagonal
	restriction.  Use
	\eqref{eq:off-diagonal-kernel-limit} when \(y_j\ne x\) and
	\eqref{eq:diagonal-kernel-limit} when \(y_j=x\).  The boundedness of
	\(\partial_j\varphi\) and of the kernels permits passage to the limit and
	gives \eqref{eq:multilevel-diagonal}.
\end{proof}

\section*{Declaration on the use of generative AI}

Generative AI tools were used for language editing, \LaTeX{} formatting,
and preliminary literature searches.  All mathematical arguments,
calculations, and references were checked by the author, who takes full
responsibility for the content of the paper.

\end{document}